\documentclass[12pt,letterpaper]{amsart}
\usepackage{amsfonts}
\usepackage{bm}
\usepackage{bbm}  
\usepackage{fullpage}
\usepackage[letterpaper, textwidth=13.5cm, centering]{geometry}
\usepackage[colorlinks=true]{hyperref}
\usepackage[extra]{tipa}
\usepackage{indentfirst}
\usepackage{graphicx}

\usepackage{amsmath,amssymb,amsfonts,amsthm}
\usepackage[utf8]{inputenc}
\usepackage[english]{babel}
\usepackage{hyperref}
\usepackage{tikz-cd}
\usepackage{enumitem}
\usepackage{comment}

\newcommand\numberthis{\addtocounter{equation}{1}\tag{\theequation}}
\theoremstyle{definition}
\newtheorem{defn}{Definition}[section]
\newtheorem{theorem}[defn]{Theorem}
\newtheorem{corollary}[defn]{Corollary}
\newtheorem{lemma}[defn]{Lemma}
\newtheorem{prop}[defn]{Proposition}
\newtheorem{remark}[defn]{Remark}
\newtheorem{example}[defn]{Example}
\numberwithin{equation}{section}

\title{The Norm Residue Symbol for Formal Drinfeld Modules}

\author{Marwa Ala Eddine$^{*}$ \and Behzad Omidi Koma}

\date{}

\begin{document}

\maketitle

\begin{abstract}
We study the Kummer pairing associated with formal Drinfeld modules of stable reduction and height one over local fields. Building on the explicit reciprocity framework initiated by Kolyvagin and extended in the context of Drinfeld modules by Anglès and Bars--Longhi, we obtain an explicit description of the pairing in terms of the logarithm of the formal module, trace maps, and explicitly constructed derivations.

Our approach yields a unified and computable formula for the norm residue symbol in this setting and extends previously known results from the Carlitz and sign-normalized rank-one cases to arbitrary finite extensions containing sufficient torsion. As an application, we derive explicit congruences for local norm maps associated with torsion points of formal Drinfeld modules.

\medskip

\noindent Keywords: Formal Drinfeld modules; explicit reciprocity laws; local fields; class field theory; derivations.

\medskip

\noindent \textit{2020 Mathematics Subject Classification: 11F85, 11S31.}
\end{abstract}
\medskip
\medskip

\small{
\noindent $^{*}$Corresponding author: Marwa Ala Eddine\\
\textit{Laboratoire de Math\'ematiques de Besan\c{c}on, Universit\'e Marie et Louis Pasteur,\\
25000 Besan\c{c}on, France}\\
\texttt{marwaalaeddine1@gmail.com}

\vspace{0.4cm}

\noindent Behzad Omidi Koma\\
\textit{Engineering College, Department of Mathematics,\\
American University of the Middle East (AUM), Egaila, Kuwait}
\section{Introduction} 

Explicit reciprocity laws provide concrete descriptions of norm residue symbols in local class field theory. In the setting of Drinfeld modules, such formulas are known in a few special cases, notably for the Carlitz module and for rank-one sign-normalized Drinfeld modules. The purpose of this paper is to obtain an explicit formula for the Kummer pairing associated with formal Drinfeld modules of stable reduction and height one over local fields, extending these results to a broader and more flexible framework.

More precisely, we express the norm residue symbol in terms of the logarithm of the formal Drinfeld module and an explicitly constructed derivation inspired by Kolyvagin’s method. This leads to a computable formula for the pairing and yields explicit congruences for local norm maps associated with torsion points.

Let $K$ be a local field, $p$ be its characteristic, and let $\mu_K$ be its normalized discrete valuation. We denote $\mathcal{O}_K$ the valuation ring of $K$ and $\mathfrak{p}_K$ its maximal ideal. Let $q$ be the order of the residue field $\mathcal{O}_K/\mathfrak{p}_K$. Then $q$ is a power of $p$. Fix an algebraic closure $\Omega$ of $K$, and let $\mu$ be the unique extension of $\mu_K$ to $\Omega$. All the extensions $F$ of $K$ considered in this paper are supposed to be such that $F \subset \Omega$. 
We also denote $\pi_F$ a uniformizer of $F$, $\mathcal{O}_F$ the valuation ring of $F$ and $\mathfrak{p}_F$ its maximal ideal.
Let $K_{ur} \subset \Omega$ be the maximal unramified extension of $K$ in $\Omega$, and $H \subset K_{ur}$ be a finite unramified extension of $K$.

Let 
\begin{align*} 
    \rho: \mathcal{O}_K & \longrightarrow \mathcal{O}_H\{\{\tau \}\} \\
    a &  \mapsto \rho_a
\end{align*}
be a formal Drinfeld module having stable reduction of height one, as defined by Rosen in \cite[\S 1]{Rosen}. Here, $\tau$ is the $q$-Frobenius element satisfying 
\begin{equation}
    \tau x= x^q \tau, ~~ \forall x \in \Omega.
\end{equation}

The formal Drinfeld module $\rho$ induces an action on the completion $\Bar{\Omega}$ of $\Omega$ as follows

\begin{equation} \label{defactionrho}
    a \cdot_{\rho} x =\rho_a(x) ~~ \forall x \in \Bar{\Omega}.
\end{equation} 
For an integer $n \ge 0$, let 
$$    V_{\rho}^n= \{ \alpha \in \Bar{\Omega} ;~ \rho_{a}(\alpha)=0 ~~\forall a \in \mathfrak{p}_K^n \} $$
be the $\mathfrak{p}_K^n$ torsion submodule of $\Bar{\Omega}$ for the action (\ref{defactionrho}). It is isomorphic as an $\mathcal{O}_K$- module to $ \mathcal{O}_K / \mathfrak{p}_K^n$.  Any element $v_0 \in V_{\rho}^n \setminus V_{\rho}^{n-1}$ is therefore a generator of $V_{\rho}^n$ and the extension $H_{\rho}^n=H(V_{\rho}^n)$ is equal to $H(v_0)$. For more details see \cite{Iwasawa,Oukhaba}.

Now let $m_0 \ge 1$ be an integer dividing $[H:K]$, and $\eta \in K$ of valuation $\mu(\eta)=m_0$.  Let 
$$
     W_{\rho}^n=V_{\rho}^{nm_0} =\{ \alpha \in \mathfrak{p}_{\Bar{\Omega}} ; ~\rho_{\eta^n}(\alpha)=0\},~~ \text{and} ~~ W_{\rho}= \bigcup_n V_{\rho}^n =\bigcup_n W_{\rho}^n.
$$

Fix once and for all a generator $v_n$ of $W_{\rho}^n$. Let $$E_{\rho}^n= H(W_{\rho}^n)=H_{\rho}^{nm_0}.$$ 
 
\noindent If $L$ is a finite extension of $E_{\rho}^n$, then we denote by
$$\Phi_L: L^{\times} \rightarrow Gal(L^{ab}|L) $$
the norm residue map. By \cite[Lemma 2.1]{Papier1}, for each $\alpha \in \mathfrak{p}_{L}$, there exists $\xi \in L^{ab}$ such that $\rho_{\eta^n}(\xi)=\alpha$. Therefore we can define the map $(~,~)_{\rho,L,n}:\mathfrak{p}_L \times L^{\times} \longrightarrow W_{\rho}^n$ such that
\begin{equation} \label{eqdefpairing}
    (\alpha,\beta)_{\rho,L,n}=\Phi_L(\beta)(\xi)-\xi; ~~\rho_{\eta^n}(\xi)=\alpha,
\end{equation}
for $\alpha \in \mathfrak{p}_{L}$ and $\beta \in L^{\times}$. We omit $\rho$ in the index when there is no risk of confusion. \newline

The objective of this paper is to find explicit formulas for the pairing \eqref{eqdefpairing}, generalizing the results those found in a previous work \cite{Papier1} and in the special cases considered by Angl\`es \cite{angles} and Bars-Longhi \cite{Ignazio}. The main ingredient we use here is the derivations and their properties, following the work of Kolyvagin \cite{kolyvagin} in the case of zero characteristic. The results we obtain are the following.

\begin{theorem} (Theorem \ref{mainTH})
    Suppose that $L|K$ is a separable extension, then there exists an $\mathcal{O}_K$-derivation $\operatorname{\Bar{D}}_{L,v_n}$ from $\mathcal{O}_L$ into a certain $\mathcal{O}_L$-submodule $W$ of $L$ such that 
\begin{equation} \label{mainequationD}
(\alpha,\beta)_{L,n}=  T_{L|K}(\lambda_{\rho}(\alpha)\operatorname{dlog\Bar{D}}_{L,v_n}(\beta)) \cdot_{\rho} v_n
\end{equation}
for all $\beta \in L^{\times}$ and $\alpha \in L$ of valuation $\mu(\alpha) > \frac{nm_0}{q} + \frac{1}{q-1} + \frac{1}{e(L|K)}$. Here, $\lambda_{\rho} $ is the logarithm of $\rho$ and $e(L|K)$ is the ramification index of $L|K$. For $\beta = u \pi_L^k \in L^{\times}$, the logarithmic  derivative $\operatorname{dlog\Bar{D}}_{L,v_n}$ associated with the derivation $\operatorname{\Bar{D}}_{L,v_n}$ is defined as follows
\begin{equation}
    \operatorname{dlog\Bar{D}}_{L,v_n}(\beta)=\dfrac{\operatorname{\Bar{D}}_{L,v_n}(u)}{u} + k \dfrac{\operatorname{\Bar{D}}_{L,v_n}(\pi_L)}{\pi_L}.
\end{equation}
\end{theorem}

An advantage of having a derivation is that it is determined and explicitly constructible in terms of its value at a uniformizer $\pi_L$ of L as follows. For $x \in \mathcal{O}_L$, we can write
\begin{equation}
    \operatorname{\Bar{D}}_{L,v_n}(x)=f'(\pi_L)\operatorname{\Bar{D}}_{L,v_n}(\pi_L),
\end{equation}
where $f$ is the unique power series in $\mathbb{F}_{q_L}[[X]]$ such that $x=f(\pi_L)$. Here, $\mathbb{F}_{q_L}$ denotes the residue field of $L$.
In the particular case where $L=E_{\rho}^m$ and $\pi_L=v_m$, our previous work in \cite{Papier1} implies the subsequent proposition.

\begin{prop} (Proposition \ref{propValueD})
Suppose $\rho$ is such that $ \rho_{\eta} \equiv \tau^{m_0} \mod \mathfrak{p}_H$. For all $m \ge n$, we have
\begin{equation}
     \operatorname{\Bar{D}}_{E_{\rho}^m,v_n}(v_m)=\frac{1}{\eta^m}.
\end{equation}
\end{prop}

Finally, using invariants attached to the representation $\mathbf{r}:\operatorname{Gal}(\Omega|H) \to \operatorname{GL}_1(\mathcal{O}_K) = \mathcal{U}_K$, which is induced by the action of $\operatorname{Gal}(\Omega|H)$ on the module $\varprojlim W_{\rho}$, we get the following congruence, which is equivalent to Proposition \ref{propValueD}, and of which we do not have a direct proof.

\begin{prop} (Corollary \ref{coroCong})
Suppose $\rho$ is such that $ \rho_{\eta} \equiv \tau^{m_0} \mod \mathfrak{p}_H$ and let $L=E_{\rho}^m$ for $m \ge n$. Let $u$ be a unit of $L$ such that $\mu(1-u) > \max\{\frac{nm_0}{q},\frac{1}{q-1}\} + \frac{1}{q-1}$. Then
\begin{equation}
    \operatorname{N}_{L|K}(u^{-1})-1 \equiv \operatorname{T}_{L|K}((\frac{1-u}{u})(1- \frac{g'(v_m)}{u}v_m))  \mod \mathfrak{p}_K^{(n+m)m_0},
\end{equation}
where $g(X) \in \mathbb{F}_{q_L}[[X]]$ is such that $g(v_m)=u$.
\end{prop}

The method used to obtain \eqref{mainequationD} was inspired by the work of Kolyvagin \cite{kolyvagin}, in which he proved explicit formulas for the Kummer pairing in the case of formal groups of finite height in zero characteristic. The results of Kolyvagin extended those of Iwasawa \cite{Iwasawa2} and Wiles \cite{Wiles}, who proved explicit laws for the Kummer pairing associated to the multiplicative group and to general Lubin-Tate formal groups respectively. The results obtained here extend those of Angl\`es \cite{angles} proved for Carlitz module, and of Bars and Longhi \cite{Ignazio} proved for sign-normalized rank one Drinfeld modules. In his turn, Florez \cite{Florez1,Florez2} followed Kolyvagin's method to generalize the latter's work and prove explicit laws in the case of formal groups and Lubin-Tate formal groups defined over arbitrary higher local field of mixed characteristic. Whence, one may ask if we can generalize our results as well to local fields of higher dimension.

The paper is organized as follows. In Section 2 we recall basic properties of the Kummer pairing. Section 3 is devoted to the construction of the associated derivation and to the proof of the main theorem.
\section{Properties of the Kummer pairing}
In this section, we state some of the main properties of the pairing $(~,~)_{L,n}$. Throughout this section, fix a positive integer $n$ and a finite extension $L$ of $E_{\rho}^n$.
\begin{prop} \label{PropoertiesPairing1}
The map $(~,~)_{L,n}$ satisfies the following properties

\begin{enumerate} [label=(\roman*)]
    \item The map $(~,~)_{L,n}$ is bilinear and $\mathcal{O}$-linear in the first coordinate for the action (\ref{defactionrho}).
    \item We have $$(\alpha,\beta)_{L,n}=0  \iff \beta  \text{ is a norm from } L(\xi)\text{, where } \rho_{\eta^n}(\xi)=\alpha.$$ 
    \item Let $M$ be a finite separable extension of $L$, let $\alpha \in \mathfrak{p}_L$ and $\beta \in M^{\times}$. Then $(\alpha,\beta )_{M,n} =(\alpha,\operatorname{N}_{M|L}(\beta))_{L,n}$.
    \item Let $M$ be a finite separable extension of $L$ of degree $d$, let $\alpha \in \mathfrak{p}_M$ and $\beta \in L^{\times}$. Then $(\alpha,\beta)_{M,n}=(\operatorname{T}_{M|L}(\alpha),\beta)_{L,n}$.
    \item Suppose $L \supset E_{\rho}^m$ for $m \ge n$. Then 
    $$  (\alpha,\beta)_{L,n} = \rho_{\eta^{m-n}}((\alpha,\beta)_{L,m} ) = (\rho_{\eta^{m-n}}(\alpha),\beta)_{L,m} .$$
    \item Let $\rho'$ be a formal Drinfeld $\mathcal{O}$-module isomorphic to $\rho$, i.e there exists a power series $t$ invertible in $\mathcal{O}_H\{\{\tau\}\}$ such that $\rho'_a= t^{-1} \circ \rho_a \circ t$ for all $a \in \mathcal{O}$. Then we have $(\alpha,\beta)_{\rho',L,n}= t^{-1}((t(\alpha),\beta)_{\rho,L,n})$.
\end{enumerate}
\end{prop}
We omit the proof of these properties. The interested reader may find a detailed proof in \cite[\S 3.3]{kolyvagin}. As mentioned in the introduction, there exists explicit formulas for the pairing $(~,~)_{L,n}$ which include the logarithm of the Drinfeld module $\rho$. This so-called logarithm was defined by Rosen in \cite[\S 2]{Rosen} as follows.
\begin{lemma} \label{lemmadefrho}
There exists a unique power series $\lambda_{\rho} \in H\{\{\tau\}\}$, called the logarithm of $\rho$, such that $\lambda_{\rho}(X) \equiv X \mod \operatorname{deg~2}$ and $\lambda_{\rho} \rho_a = a \lambda_{\rho}$ for all $a \in \mathcal{O}$. Moreover, we have
\begin{enumerate}[label=(\roman*)]
    \item \label{coefflamda} If $\lambda_{\rho}=\sum_{i \ge 0} c_i \tau^i$, then $ \mu(c_i) \ge -i$ for all $i \ge 0$. Thus the element $\lambda_{\rho}(x)=\sum_{i \ge 0} c_i x^{q^i}$ is well defined in $L$ for any $x \in \mathfrak{p}_L$.
    \item If $x \in \mathfrak{p}_{\Omega}$, then $\lambda_{\rho}(X)=0$ if and only if $x \in V_{\rho}$. Put $W_L=L \cap W_{\rho} \subset \mathfrak{p}_L$. Then the map $\lambda_{\rho}: \mathfrak{p}_L/W_L \longrightarrow  \lambda_{\rho}(\mathfrak{p}_L)$ is an isomorphism of $\mathcal{O}$-modules.
    \item Let $\mathfrak{p}_{\Omega,1}$ denote the set of all the elements $x$ of $\mathfrak{p}_{\Omega}$ such that $\mu(x) > 1/(q-1)$. The logarithm $\lambda_{\rho}$ gives an isomorphism of $\mathcal{O}$-modules from $\mathfrak{p}_{\Omega,1}$, viewed as an $\mathcal{O}$-module under the action (\ref{defactionrho}), to itself, viewed as an $\mathcal{O}$-module under the multiplication in $\Omega$. In particular, if we denote $\mathfrak{p}_{L,1} = \mathfrak{p}_{L} \cap \mathfrak{p}_{\Omega,1}$, the logarithm $\lambda_{\rho}$ also induces an isomorphism from the ideal $\mathfrak{p}_{L,1}$ to itself. This follows from the fact that $\mu(\lambda_{\rho}(x))=\mu(x)$ for all $x \in \mathfrak{p}_{\Omega,1}$. 
\end{enumerate}
\end{lemma}
Inspired by \cite{kolyvagin}, we proved in \cite[\S 3]{Papier1} the following explicit formula for $(~,~)_{L,n}$. 
We denote by $\mathfrak{X}_{L,1} \subset L$ the fractional ideal of all elements $y$ such that $\operatorname{T}_{L|K}(xy) \in \mathcal{O}_K$ for all $ x \in \lambda_{\rho}(\mathfrak{p}_{L,1})$. We have
\begin{align*}
    \mathfrak{X}_{L,1} & = \{ y \in L;~~\operatorname{T}_{L|K}(\lambda_{\rho}(\alpha)y) \in \mathcal{O}_K ~~ \forall \alpha \in \mathfrak{p}_{L,1} \} \\
    & = \{  y \in L;~~ \operatorname{T}_{L|K}(\alpha' y) \in \mathcal{O}_K ~~ \forall \alpha' \in \mathfrak{p}_{L,1} \} \numberthis \label{defX_L1} \\
    & = \{  y \in L;~~ \mu(y) \ge -\frac{1}{q-1} - \frac{1}{e(L|K)} -\mu(\mathcal{D}_{L|K})\}.
\end{align*}
\begin{prop} \label{iwasawamappsi}
Suppose that the extension $L|K$ is separable. Then, there exists a unique map $\Psi_{L,v_n}: L^{\times} \longrightarrow \mathfrak{X}_{L,1}/\eta^n\mathfrak{X}_{L,1}$ 
such that 
\begin{equation} \label{eqiwasawamappsi}
    (\alpha,\beta)_{L,n}=  T_{L|K}(\lambda_{\rho}(\alpha)\Psi_{L,v_n}(\beta)) \cdot_{\rho} v_n 
\end{equation}
for all $\alpha \in \mathfrak{p}_{L,1}$ and $\beta \in L^{\times}$. Furthermore, $\Psi_{L,v_n}$ is a continuous group homomorphism.
\end{prop}
\begin{remark}
\begin{enumerate} [label=(\roman*)]
    \item In \eqref{eqiwasawamappsi}, we view $\Psi_{L,v_n}(\beta)$ as an element of $\mathfrak{X}_{L,1}$. It it easy to see that for any $\alpha \in \mathfrak{p}_{L,1}$, the value of $ T_{L|K}(\lambda_{\rho}(\alpha)\Psi_{L,v_n}(\beta)) \cdot_{\rho} v_n$ does not depend on the choice of the representative of the class of $\Psi_{L,v_n}(\beta)$ in $\mathfrak{X}_{L,1}/\eta^n\mathfrak{X}_{L,1}$.
    \item Let $v'_n$ be another generator of $W_{\rho}^n$, then $v'_n = \rho_u(v_n)$ for a unit $u$ of $K$. We have
    \begin{equation}
        \Psi_{L,v_n}=u \Psi_{L,v'_n}.
    \end{equation}
\end{enumerate}
    
\end{remark}
Exactly as in \cite[\S 3.5]{kolyvagin}, our $\Psi_{L,v_n}$ satisfies the properties $\varphi_1,~ \varphi_2,~ \varphi_3,~ \varphi_4,~ \varphi_5$ and $\varphi_6$ of loc. cit. The equality \eqref{eqiwasawamappsi} gives an expression of the pairing $(~,~)_{L,n}$ in terms of the trace of $L|K$, the logarithm of $\rho$, and the map $\Psi_{L,v_n}$. However, we do not have an explicit expression of $\Psi_{L,v_n}$. Therefore, we will use $\Psi_{L,v_n}$ to construct a derivation $\operatorname{\Bar{D}}_{L,v_n}$ (see \S \ref{sectionderivationD} below), which will help us prove an explicit formula for $(~,~)_{L,n}$. In fact, we will see that $\operatorname{\Bar{D}}_{L,v_n}$ is determined by its value at a prime $\pi_L$, which is, by its turn, determined by invariants from representation theory.
\begin{prop} \label{existencer}
There exists a unique power series $r=r_n \in \mathcal{O}_H\{\{\tau\}\}$ such that $$ \prod_{\omega \in W_{\rho}^n}(X-\omega)=r \circ \rho_{\eta^n}(X) .$$ 
Furthermore, the power series $r$ is invertible in $\mathcal{O}_H\{\{\tau\}\}$ and satisfies

$$(x,r(x))_{L,n}=0, ~~ \forall x \in \mathfrak{p}_L  .$$
\end{prop}
\begin{proof}
    See \cite[Proposition 4.3]{Papier1} and \cite[Lemma 17]{Ignazio}.
\end{proof}
\begin{lemma} \label{condrhoprime}
Let $r=r_n$ be the power series defined in Proposition \ref{existencer}. Let $\rho'$ be  defined by 
\begin{equation} \label{defrhoprime}
    \rho'_a= r \circ \rho_a \circ r^{-1}
\end{equation}
for all $a \in \mathcal{O}$. Then $\rho'$ is a formal Drinfeld module having a stable reduction of height 1, and we have 
\begin{equation} \label{propxx}
    (x,x)_{\rho',L,n}=0~~\text{for all  } x \in \mathfrak{p}_L.
\end{equation}
\end{lemma}
\begin{proof}
    See \cite[Lemma 4.4]{Papier1}.
\end{proof}
As we will see in the sequel, it will be easier to deal with formal Drinfeld modules satisfying the property \eqref{propxx}. Lemma \ref{condrhoprime} will ensure that, starting any formal Drinfeld module having a stable reduction of height 1, we will be able to reach, by isomorphism, a formal Drinfeld module having a stable reduction of height 1, satisfying \eqref{propxx}.
\section{Derivations}

\subsection{Recall on Derivations}
In this paragraph, we give a brief recall on derivations and their main properties that will be useful for us in the sequel. Let $R$ be a commutative ring with unit, and $O$ be a subring of $R$. If $W$ is an $R$-module, a map $D: R \to W$ is said to be an $O$-derivation of $R$ into $W$ if it is $O$-linear and satisfies the Leibniz rule
\begin{equation}
    D(xy) = x D(y) + y D(x) ~~ \forall x, y \in R.
\end{equation}
In particular, a derivation $D: R \to W$ also fulfills the following:
\begin{enumerate} [label=(\roman*)]
    \item $D(x+y) = D(x) + D(y) ~~ \forall x, y \in R,$
    \item $D(a)=0 ~~ \forall a \in O$.
\end{enumerate}
The set of all such derivations $\operatorname{D}_{O}(R,W)$ is an $R$-module, where $aD$ is defined by $(aD)(x)=aD(x)$ for all $a,x \in R$. We will show that there exists a universal derivation, in other words, an $R$-module $\Omega_{O}(R)$, and a derivation 
\begin{equation}
    d:R \to \Omega_{O}(R)
\end{equation}
such that for every derivation $D:R \to W$, there exists a unique homomrphism of $R$-modules $f: \Omega_{O}(R) \to W$ such that the diagram
\begin{center}
\begin{tikzcd}[column sep=small]
R \arrow{r}{d}  \arrow[swap]{rd}{D} 
  & \Omega_{O}(R) \arrow[dashed]{d}{\exists ! f} \\
    & W
\end{tikzcd}    
\end{center}
commuts.
Let $\mathcal{R}$ be the direct sum of the modules $(R)_{x \in R}$.Then $\mathcal{R}$ is the submodule of $\prod_{x \in R}R$ which consists of families $(a_x)_{x \in R}$ having finite support. For each element $x \in R$, we associate a symbol $\operatorname{d}x$,  so that an element $(a_x)_{x \in R}$ in $\mathcal{R}$ can be written as a finite sum $\sum_{x \in R}a_x \operatorname{d}x$. Here, the symbols $\operatorname{d}x$ are supposed to be distinct for distinct elements of $R$. Consider the submodule of $\mathcal{R}$ generated by the set 
\begin{equation} \label{setderiv}
    \{ \operatorname{d}(xy) - y\operatorname{d}x - x \operatorname{d}y,~~ \operatorname{d}(x+y) -\operatorname{d}x -\operatorname{d}y,~~ \operatorname{d}a; ~~ x, y \in R, ~ a \in O \}.
\end{equation}
The quotient of $\mathcal{R}$ by this submodule , which we denote by $\Omega_O(R)$, together with the derivation $\operatorname{d}: R \to \Omega_O(R)$ that sends $x$ to the class of $\operatorname{d}x$ in $\Omega_O(R)$, form the universal derivation we are looking for. Indeed, let $W$ be an $R$-module and $D:R \to W$ be a derivation, and consider the unique homomorphism of $R$-modules from $\mathcal{R}$ to $W$ that maps $\operatorname{d}a$ to $D(a)$. This homomorphism is trivial on the submodule of $\mathcal{R}$ generated by the set \eqref{setderiv}, thus it factors through $\Omega_O(R)$, whence the universal property. We call $\Omega_O(R)$ the module of differentials of $R$ over $O$.

The universal derivation yields an isomorphism of $R$-modules
\begin{equation}
    \operatorname{D}_{O}(R,W) \simeq \operatorname{Hom}_R(\Omega_O(R),W).
\end{equation}

Let $M$ be a local field and $N$ be a finite separable extension of $M$. We denote by $\mathcal{D}(N|M)$ the different of $N|M$. In the special case where $R=\mathcal{O}_N$ and $O=\mathcal{O}_M$, we have the following results.

\begin{prop} \label{structureOmega}
There exists an isomorphism of $\mathcal{O}_N$-modules
\begin{equation}
    \Omega_{\mathcal{O}_M}(\mathcal{O}_N) \simeq \mathcal{O}_N / \mathcal{D}(N|M).
\end{equation}
Furthermore, if $\pi_N$ is a prime of $N$, then $d\pi_N$ is a generator of $\Omega_{\mathcal{O}_M}(\mathcal{O}_N)$.
\end{prop}

\begin{proof}
    Kolyvagin proved this proposition in the case of zero characteristic \cite[Proposition 5.1]{kolyvagin}. His proof is suitable for our case.
\end{proof}
\begin{corollary} \label{corostructureOmega}
Let $W$ be an $\mathcal{O}_N$-module and $\pi_N$ be a prime of $N$. Let 
\begin{equation}
    S:=\{x \in W,~~ax=0~~\forall a \in \mathcal{D}(N|M)\}
\end{equation}
be the $\mathcal{D}(N|M)$-torsion submodule of $W$. Then, the map 
\begin{align} \label{isomderiv}
    \operatorname{D}_{\mathcal{O}_M}(\mathcal{O}_N,W) & \to S \\
    D & \mapsto D(\pi_N) \nonumber
\end{align}
is an isomorphism of $\mathcal{O}_N$-modules. 
\end{corollary}

\begin{proof}
   The proof of \cite[Corollary 5.2]{kolyvagin} is convenient for our case as well.
\end{proof}
\begin{remark}
With the notations of Corollary \ref{corostructureOmega}, the inverse homomorphism of \eqref{isomderiv} associates to an element $x \in S$ a derivation $\operatorname{D}_x$ satisfying
\begin{equation}
    D_x(t(\pi_N)) = t'(\pi_N)x
\end{equation}
for all $t \in \mathcal{O}_{\tilde{N}}[[X]]$, where $\tilde{N}$ is the inertia field of $N|M$. This follows from the fact that a derivation in $\operatorname{D}_{\mathcal{O}_M}(\mathcal{O}_N,W)$ is continuous for the discrete topology on $W$.
\end{remark}

\subsection{The derivation $\operatorname{\Bar{D}}_{L,v_n}$} \label{sectionderivationD}
In this section, we assume that $L|K$ is a separable extension. We define the map 
$\operatorname{D}_{L,v_n}: \mathcal{O}_L \longrightarrow \mathfrak{X}_{L,1} / \eta^n \mathfrak{X}_{L,1}$ by $\operatorname{D}_{L,v_n}(0)=0$ and
$D_{L,v_n}(\alpha) = \alpha \Psi_{L,v_n}(\alpha)$ for $\alpha \in \mathcal{O}_L \setminus \{0\}$, where $\Psi_{L,v_n}$ is the homomorphism defined in  \eqref{iwasawamappsi}. In this section we will prove that $\operatorname{D}_{L,v_n}$, reduced modulo a convenient submodule of $\mathfrak{X}_{L,1}$, is a derivation, and it satisfies \eqref{mainequationD}.

It is clear that the map $\operatorname{D}_{L,v_n}$ satisfies the Leibniz rule
\begin{equation} \label{leibniz}
    \operatorname{D}_{L,v_n}(xy) = x \operatorname{D}_{L,v_n}(y) + y \operatorname{D}_{L,v_n}(x) ~~ \forall x, y \in \mathcal{O}_L.
\end{equation}
This follows from the fact that $\Psi_{L,v_n}$ is a group homomorphism. Using this rule, we can prove by induction that
\begin{equation} \label{powerp}
    \operatorname{D}_{L,v_n}(x^m)= m x^{m-1} \operatorname{D}_{L,v_n}(x) ~~ \forall x \in \mathcal{O}_L ~~\text{and }~~\forall m \ge 1.
\end{equation}
We will now prove that $\operatorname{D}_{L,v_n}$ is additive.
\begin{lemma} \label{alphauu}
Suppose $\rho$ is such that $(x,x)_{\rho,L,n}=0$ for all $x \in \mathfrak{p}_L$. Let $\alpha \in \mathfrak{p}_L \setminus \{0\}$ and let $u$ be a unit of $L$ such that $\mu(\alpha(1-u)) > \frac{nm_0}{q} + \frac{1}{q-1}$. We have 
\begin{equation} \label{alphau}
    (\alpha u, u)_{L,n} = \operatorname{T}_{L|K}((1-u)\operatorname{D}_{L,v_n}(\alpha)) \cdot_{\rho} v_n.
\end{equation}
\end{lemma}

\begin{proof}
We have
\begin{align*}
    (\alpha u,u)_{L,n} & = (\alpha u,\frac{\alpha u}{\alpha})_{L,n} \\
    & = (\alpha u, \alpha u)_{L,n} - (\alpha u,\alpha)_{L,n} \\
    & = (\alpha ,\alpha)_{L,n} - (\alpha u,\alpha)_{L,n} \\
    & = (\alpha - \alpha u,\alpha)_{L,n} \\
    & = \operatorname{T}_{L|K} (\lambda_{\rho}(\alpha-\alpha u) \Psi_{L,v_n}(\alpha))\cdot_{\rho} v_n
\end{align*}
by Proposition \ref{iwasawamappsi}. Let $\gamma = \alpha(1-u)$, we will show that 
\begin{equation}
    \operatorname{T}_{L|K} (\lambda_{\rho}(\gamma) \Psi_{L,v_n}(\alpha))\cdot_{\rho} v_n = \operatorname{T}_{L|K} (\gamma \Psi_{L,v_n}(\alpha))\cdot_{\rho} v_n.
\end{equation}
By the hypothesis on the valuations, we have $\mu(\gamma) > \frac{nm_0}{q} + \frac{1}{q-1}$. Hence
\begin{align*}
    \mu(\lambda_{\rho}(\gamma) - \gamma) & = \mu(\sum_{i \ge 1} c_i \gamma^{q^i}) \\
    & \ge \min_{i \ge 1 } \{ \mu(c_i) + q^i \mu(\gamma) \} \\
    & > \min_{i \ge 1 } \{ -i +q^i (\frac{nm_0}{q} + \frac{1}{q-1}) \} \\
    & \ge nm_0 + \frac{1}{q-1}
\end{align*}
Therefore, we can write $\lambda_{\rho}(\gamma) - \gamma = \eta^n \delta$, where $\delta$ is an element of $\mathfrak{p}_{L,1}$. Thus, by \eqref{defX_L1},
\begin{equation*}
    \operatorname{T}_{L|K}((\lambda_{\rho}(\gamma) - \gamma) \Psi_{L,v_n}(\alpha)) \cdot_{\rho} v_n= 0
\end{equation*}
because $\Psi_{L,v_n}(\alpha) \in \mathfrak{X}_{L,1}$. This concludes the proof.
\end{proof}
\begin{prop} \label{D'additive}
Suppose $\rho$ is such that $(x,x)_{\rho,L,n}=0$ for all $x \in \mathfrak{p}_L$. Let $\gamma$ be an element of $\mathcal{O}_L \setminus \{0\}$ of valuation $\mu(\gamma) = \max\{\frac{nm_0}{q},\frac{1}{q-1}\}$, that is $\mu(\gamma) = \frac{nm_0}{q}$ if $nm_0 \ge 2$, and  $\mu(\gamma) = \frac{1}{q-1}$ if $nm_0=1$. Then 
\begin{equation} \label{D'int1}
\operatorname{D}_{L,v_n}(x+y) \equiv \operatorname{D}_{L,v_n}(x)+\operatorname{D}_{L,v_n}(y) \mod{ \frac{\eta^n}{\gamma}\mathfrak{X}_{L,1}}
\end{equation}
for all $x,y \in \mathcal{O}_L$. 
\end{prop}

\begin{proof}
Let us prove first why such a $\gamma$ exists. Since $E_{\rho}^n \subset L$, the ramification index of $L|K$ is a multiple of the ramification index of $E_{\rho}^n|K$, which is equal to $q^{nm_0-1}(q-1)$. Hence, there exists elements in $L$ of valuation $\frac{1}{q^{nm_0-1}(q-1)}$, whence the existence of $\gamma$. Now let us prove \eqref{D'int1}.
Let $x, y \in \mathcal{O}_L$, then, by Lemma \ref{alphauu}, we have
\begin{align*}
    (\gamma(x+y) u, u)_{L,n} & = \operatorname{T}_{L|K}((1-u)\operatorname{D}_{L,v_n}(\gamma(x+y))) \cdot_{\rho} v_n \\
    & = \operatorname{T}_{L|K}((1-u)((x+y)\operatorname{D}_{L,v_n}(\gamma)+ \gamma \operatorname{D}_{L,v_n}((x+y))) \cdot_{\rho} v_n \numberthis \label{plusD1}
\end{align*}
for all $u \in 1+ \mathfrak{p}_{L,1}$.
However, again by Lemma \ref{alphauu}, we have
\begin{align*}
    (\gamma (x+y)u, u)_{L,n} & = (\gamma x u, u)_{L,n}+(\gamma y u, u)_{L,n} \\
    & = \operatorname{T}_{L|K}((1-u)\operatorname{D}_{L,v_n}(\gamma x )) \cdot_{\rho} v_n + \operatorname{T}_{L|K}((1-u)\operatorname{D}_{L,v_n}(\gamma y )) \cdot_{\rho} v_n \\
    & = \operatorname{T}_{L|K}((1-u)(\operatorname{D}_{L,v_n}(\gamma x)+\operatorname{D}_{L,v_n}(\gamma y ))) \cdot_{\rho} v_n\\
    & = \operatorname{T}_{L|K}((1-u)((x+y)\operatorname{D}_{L,v_n}(\gamma)+ \gamma (\operatorname{D}_{L,v_n}(x)+\operatorname{D}_{L,v_n}(y))) \cdot_{\rho} v_n \numberthis \label{plusD2}
\end{align*}
for all $u \in 1+ \mathfrak{p}_{L,1}$. Therefore, \eqref{plusD1} and \eqref{plusD2} being equal, we conclude that 
\begin{equation}
    \gamma \operatorname{D}_{L,v_n}(x+y) \equiv \gamma (\operatorname{D}_{L,v_n}(x)+\operatorname{D}_{L,v_n}(y)) \mod{\eta^n \mathfrak{X}_{L,1}}
\end{equation}
by the very definition \eqref{defX_L1} of $\mathfrak{X}_{L,1}$. Hence, we have
\begin{equation}
    \operatorname{D}_{L,v_n}(x+y) \equiv  \operatorname{D}_{L,v_n}(x)+\operatorname{D}_{L,v_n}(y) \mod{\frac{\eta^n}{\gamma} \mathfrak{X}_{L,1}}.  
\end{equation}
\end{proof}
\begin{corollary} \label{Dadditive}
Let $\gamma$ be as in Proposition \ref{D'additive}. Then 
\begin{equation}
\operatorname{D}_{L,v_n}(x+y) \equiv \operatorname{D}_{L,v_n}(x)+\operatorname{D}_{L,v_n}(y) \mod{ \frac{\eta^n}{\gamma}\mathfrak{X}_{L,1}}
\end{equation}
for all $x,y \in \mathcal{O}_L$.     
\end{corollary}

\begin{proof}
Let $r$ be the series defined in Proposition \ref{existencer} and let $\rho'$ the Drinfeld module defined by 
$$  \rho'_a= r \circ \rho_a \circ r^{-1}. $$
Then $r$ defines an isomorphism of $\mathcal{O}_K$-modules $r: W_{\rho}^n \to W_{\rho'}^n$. Furthermore, if we denote by $\operatorname{D}_{\rho,L,v_n}$ (respectively $\operatorname{D}_{\rho',L,r(v_n)}$) the map defined in the beginning of \S \ref{sectionderivationD} associated to $\rho$ (respectively $\rho'$), we have
\begin{equation} \label{D-D'}
     \operatorname{D}_{\rho,L,v_n}= r'(0) \operatorname{D}_{\rho',L,r(v_n)}.
\end{equation}
Here, $r'(0)$ is a  unit in $H$ because $r(X) \in \mathcal{O}_H[[X]]$ is invertible. Since $(x,x)_{\rho',L,n}=0$ for all $x \in \mathfrak{p}_L$ by Lemma \ref{condrhoprime}, we can apply Proposition \ref{D'additive} for $\rho'$ so that
\begin{equation}
\operatorname{D}_{\rho',L,r(v_n)}(x+y) \equiv \operatorname{D}_{\rho',L,r(v_n)}(x)+\operatorname{D}_{\rho',L,r(v_n)}(y) \mod{ \frac{\eta^n}{\gamma}\mathfrak{X}_{L,1}}
\end{equation}
for all $x,y \in \mathcal{O}_L$. Thus, using \eqref{D-D'}, we conclude that 
\begin{equation}
\operatorname{D}_{\rho,L,v_n}(x+y) \equiv \operatorname{D}_{\rho,L,v_n}(x)+\operatorname{D}_{\rho,L,v_n}(y) \mod{ \frac{\eta^n}{\gamma}\mathfrak{X}_{L,1}}
\end{equation}
for all $x,y \in \mathcal{O}_L$.   
\end{proof}
\begin{prop}\label{propdefDbar}
Let
\begin{equation}
    \mathfrak{X}_{L,1}^{(n)} = \{ y \in L; ~ \mu(y) \ge nm_0 - \max\{\frac{nm_0}{q},\frac{1}{q-1}\} - \frac{1}{q-1} - \frac{1}{e(L|K)} - \mu(\mathcal{D}(L|K)) \} \subset \mathfrak{X}_{L,1}.
\end{equation}

The reduction of $\operatorname{D}_{L,v_n}$ modulo $\mathfrak{X}_{L,1}^{(n)}$, denoted by $\Bar{\operatorname{D}}_{L,v_n}:\mathcal{O}_L \longrightarrow \mathfrak{X}_{L,1} / \mathfrak{X}_{L,1}^{(n)}$, is an $\mathcal{O}_K$-derivation. 
\end{prop}
\begin{proof}
Let $\gamma \in \mathcal{O}_L \setminus \{0\}$ be as in Proposition \ref{D'additive}, then 
\begin{equation}
    \mathfrak{X}_{L,1}^{(n)}= \frac{\eta^n}{\gamma}\mathfrak{X}_{L,1}.
\end{equation}
Let $\pi_L$ be a prime of $L$ and let $w= \Bar{\operatorname{D}}_{L,v_n}(\pi_L) \in \mathfrak{X}_{L,1} /\mathfrak{X}_{L,1}^{(n)}$. Since $\mu(\frac{\eta^n}{\gamma}) = nm_0-\mu(\gamma) \le nm_0 - \frac{1}{q-1} \le \mu(\mathcal{D}(L|K))$, we have $\mathcal{D}(L|K)w =0$. Hence, by Corollary \ref{corostructureOmega}, there exists a derivation $\operatorname{D}:\mathcal{O}_L \longrightarrow  \mathfrak{X}_{L,1} / \frac{\eta^n}{\gamma} \mathfrak{X}_{L,1}$ such that $\operatorname{D}(\pi_L) = w$ and 
\begin{equation} \label{Dg}
    \operatorname{D}(g(\pi_L)) = g'(\pi_L) w
\end{equation}
for every power series $g \in \mathcal{O}_{\tilde{L}}[[X]]$, where $\tilde{L}$ is the maximal subextension of $L$
unramified over $K$. In particular, \eqref{Dg} is true for all the power series defined over the residue field of $\tilde{L}$, which is equal to the residue field of $L$. We will prove that $D$ and $\Bar{\operatorname{D}}_{L,v_n}$ are equal. Indeed, let $x \in \mathcal{O}_L$, and let $g (X) = \sum_{i \ge 0} a_i X^i$ be the unique power series defined over the residue field $\mathbb{F}_L$ of $L$ such that $g(\pi_L) = x$. We have
\begin{equation} \label{d=d1}
    \Bar{\operatorname{D}}_{L,v_n}(x) = \Bar{\operatorname{D}}_{L,v_n}(g(\pi_L)) = \sum_{i \ge 0} \Bar{\operatorname{D}}_{L,v_n}(a_i\pi_L^i)
\end{equation}
because $\Bar{\operatorname{D}}_{L,v_n}$ is additive by Proposition \ref{Dadditive}, and continuous by Proposition \ref{iwasawamappsi}. Let $q_L$ be the cardinal of $\mathbb{F}_{q_L}$, then $q_L$ is a power of $p$. Hence, for all $i \ge 0$, we have $$\Bar{\operatorname{D}}_{L,v_n}(a_i)=\Bar{\operatorname{D}}_{L,v_n}(a_i^{q_L})=0$$ 
by \eqref{powerp}. Therefore, applying the Leibniz rule \eqref{leibniz} to \eqref{d=d1}, we get
\begin{equation}
    \Bar{\operatorname{D}}_{L,v_n}(x)= \sum_{i \ge 0} a_i \Bar{\operatorname{D}}_{L,v_n}(\pi_L^i)= \sum_{i \ge 0} a_i \times i \times  \pi_L^{i-1} \times \Bar{\operatorname{D}}_{L,v_n}(\pi_L)
\end{equation}
again by \eqref{powerp}. However, this is equal to $g'(\pi_L) \Bar{\operatorname{D}}_{L,v_n}(\pi_L)$, which is, by \eqref{Dg}, equal to $\operatorname{D}(x)$.
\end{proof}
Now, we will define the logarithmic derivative $\operatorname{dlog\Bar{D}}_{L,v_n}$ associated to the derivation $\Bar{\operatorname{D}}_{L,v_n}$ as follows. Let
\begin{equation}
   f: \mathfrak{X}_{L,1} / \mathfrak{X}_{L,1}^{(n)}  \to  \pi_L^{-1}\mathfrak{X}_{L,1} / \pi_L^{-1}\mathfrak{X}_{L,1}^{(n)}
\end{equation}
be the natural map induced by the inclusion $\mathfrak{X}_{L,1} \hookrightarrow  \pi_L^{-1}\mathfrak{X}_{L,1} $, and 
\begin{equation}
   g_{\pi_L}: \mathfrak{X}_{L,1} / \mathfrak{X}_{L,1}^{(n)}  \to \pi_L^{-1}\mathfrak{X}_{L,1} / \pi_L^{-1}\mathfrak{X}_{L,1}^{(n)} 
\end{equation}
be the multiplication by $\pi_L^{-1}$ map. For $x= u \pi_L^k \in L^{\times}$, where $u$ is a unit in $L$, we define
\begin{equation}
    \operatorname{dlog\Bar{D}}_{L,v_n}(x)= f(u^{-1} \Bar{\operatorname{D}}_{L,v_n}(u)) + k  g_{\pi_L}(\Bar{\operatorname{D}}_{L,v_n}(\pi_L)).
\end{equation}
The map $\operatorname{dlog\Bar{D}}_{L,v_n}: L^{\times} \to \pi_L^{-1}\mathfrak{X}_{L,1} / \pi_L^{-1}\mathfrak{X}_{L,1}^{(n)}$ is a group homomorphism. Furthermore, its definition does not depend on the choice of the uniformizer $\pi_L$. Indeed, let $\pi_L'$ be another uniformizer of $L$ and let $x=u \pi_L^k = u' {\pi'_L}^k \in L^{\times}$, where $u$ and $u'$ are units of $L$. Let $u_0$ be the unit of $L$ such that $\pi_L'=u_0 \pi_L$. Then,
\begin{align*}
    f(u^{-1} \Bar{\operatorname{D}}_{L,v_n}(u)) + k  g_{\pi_L}(\Bar{\operatorname{D}}_{L,v_n}(\pi_L))
    & = f({u'}^{-1} u_0^{-k} \Bar{\operatorname{D}}_{L,v_n}(u' u_0^k)) + k  g_{\pi_L}(\Bar{\operatorname{D}}_{L,v_n}(\pi_L)) \\
    & = f({u'}^{-1} \Bar{\operatorname{D}}_{L,v_n}(u')+ u_0^{-k} \Bar{\operatorname{D}}_{L,v_n}(u_0^k)) +  k  g_{\pi_L}(\Bar{\operatorname{D}}_{L,v_n}(\pi_L)) \\
    & =f({u'}^{-1} \Bar{\operatorname{D}}_{L,v_n}(u'))+ f(u_0^{-k} \Bar{\operatorname{D}}_{L,v_n}(u_0^k)) +  k  g_{\pi_L}(\Bar{\operatorname{D}}_{L,v_n}(\pi_L)) \\
    & = f({u'}^{-1} \Bar{\operatorname{D}}_{L,v_n}(u')) + k f( u_0^{-1} \Bar{\operatorname{D}}_{L,v_n}(u_0)) +  k  g_{\pi_L}(\Bar{\operatorname{D}}_{L,v_n}(\pi_L)). \numberthis \label{dlogdef1}
\end{align*}
On the other hand, we have
\begin{align*}
    g_{\pi_L'}( \Bar{\operatorname{D}}_{L,v_n}(\pi_L')) & = g_{\pi_L'}( \Bar{\operatorname{D}}_{L,v_n}(u_0 \pi_L)) \\
    & = g_{\pi_L'}( u_0 \Bar{\operatorname{D}}_{L,v_n}(\pi_L) + \pi_L \Bar{\operatorname{D}}_{L,v_n}(u_0)) \\
    & \equiv {\pi_L'}^{-1} (u_0 \Bar{\operatorname{D}}_{L,v_n}(\pi_L) + \pi_L \Bar{\operatorname{D}}_{L,v_n}(u_0)) \mod{\pi_L^{-1}\mathfrak{X}_{L,1}^{(n)}} \\ 
    & \equiv u_0^{-1} \pi_L^{-1} (u_0 \Bar{\operatorname{D}}_{L,v_n}(\pi_L) + \pi_L \Bar{\operatorname{D}}_{L,v_n}(u_0)) \mod{\pi_L^{-1}\mathfrak{X}_{L,1}^{(n)}} \\
    & \equiv \pi_L^{-1} \Bar{\operatorname{D}}_{L,v_n}(\pi_L) +u_0^{-1} \Bar{\operatorname{D}}_{L,v_n}(u_0) \mod{\pi_L^{-1}\mathfrak{X}_{L,1}^{(n)}} \\
    & = f( u_0^{-1} \Bar{\operatorname{D}}_{L,v_n}(u_0)) + g_{\pi_L}(\Bar{\operatorname{D}}_{L,v_n}(\pi_L)). \numberthis \label{dlogdef2}
\end{align*}
Therefore, \eqref{dlogdef1} and \eqref{dlogdef2} yield that $\operatorname{dlog\Bar{D}}_{L,v_n}$ does not depend on the choice of $\pi_L$.

\begin{theorem} \label{mainTH}
The derivation $\Bar{\operatorname{D}}_{L,v_n}:\mathcal{O}_L \longrightarrow \mathfrak{X}_{L,1} / \mathfrak{X}_{L,1}^{(n)}$ satisfies
\begin{equation} \label{equalitylogder}
(\alpha,\beta)_{L,n}=  T_{L|K}(\lambda_{\rho}(\alpha)\operatorname{dlog\Bar{D}}_{L,v_n}(\beta)) \cdot_{\rho} v_n
\end{equation}
for all $\alpha$ such that $\mu(\alpha) > \max\{\frac{nm_0}{q},\frac{1}{q-1}\} + \frac{1}{q-1} + \frac{1}{e(L|K)}$ and for all $\beta\in L^{\times}$.
\end{theorem}

\begin{proof}
To prove \eqref{equalitylogder} is equivalent to prove that 
\begin{equation}
    \operatorname{dlog\Bar{D}}_{L,v_n}(\beta) - \Psi_{L,v_n}(\beta) \in \pi_L^{-1}\mathfrak{X}_{L,1}^{(n)}
\end{equation}
for all $\beta \in L^{\times}$, where $\operatorname{dlog\Bar{D}}_{L,v_n}(\beta)$ and $\Psi_{L,v_n}(\beta)$ are regarded as elements of $\pi_L^{-1}\mathfrak{X}_{L,1}$. Indeed, let $\beta \in L^{\times}$. Since Proposition \ref{iwasawamappsi} shows that
\begin{equation}
    (\alpha,\beta)_{L,n}=  T_{L|K}(\lambda_{\rho}(\alpha)\Psi_{L,v_n}(\beta)) \cdot_{\rho} v_n
\end{equation}
for all $\alpha \in \mathfrak{p}_{L,1}$, then \eqref{equalitylogder} is equivalent to say that
\begin{equation} \label{th.int.1}
    T_{L|K}(\lambda_{\rho}(\alpha)\operatorname{dlog\Bar{D}}_{L,v_n}(\beta)) \cdot_{\rho} v_n = T_{L|K}(\lambda_{\rho}(\alpha)\Psi_{L,v_n}(\beta)) \cdot_{\rho} v_n
\end{equation}
for all $\alpha$ in $L$ such that $\mu(\alpha) > \max\{\frac{nm_0}{q},\frac{1}{q-1}\} + \frac{1}{q-1} + \frac{1}{e(L|K)}$. Obviously, \eqref{th.int.1} is equivalent to 
\begin{equation} \label{th.int.2}
    T_{L|K}(\lambda_{\rho}(\alpha)(\operatorname{dlog\Bar{D}}_{L,v_n}(\beta)-\Psi_{L,v_n}(\beta)) \in \eta^n \mathcal{O}_K
\end{equation}
for all $\alpha \in \gamma \pi_L \mathfrak{p}_{L,1}$, where $\gamma \in L$ is of valuation $\mu(\gamma) = \max\{\frac{nm_0}{q},\frac{1}{q-1}\}$. However, since $\mathfrak{p}_{L,1} = \lambda_{\rho}(\mathfrak{p}_{L,1})$ and $\mu(\lambda_{\rho}(\alpha))=\mu(\alpha)$ whenever $\alpha \in \mathfrak{p}_{L,1}$ (see Lemma \ref{lemmadefrho}), then \eqref{th.int.2} is in turn equivalent to 
\begin{equation} \label{th.int.3}
    T_{L|K}(\gamma \pi_L \alpha (\operatorname{dlog\Bar{D}}_{L,v_n}(\beta)-\Psi_{L,v_n}(\beta)) \in \eta^n \mathcal{O}_K
\end{equation}
for all $\alpha \in \mathfrak{p}_{L,1}$. Finally, by the very definition of $\mathfrak{X}_{L,1}$, \eqref{th.int.3} is equivalent to 
\begin{equation} \label{th.int.4}
     \operatorname{dlog\Bar{D}}_{L,v_n}(\beta) - \Psi_{L,v_n}(\beta) \in \pi_L^{-1} \frac{\eta^n}{\gamma} \mathfrak{X}_{L,1} =\pi_L^{-1}\mathfrak{X}_{L,1}^{(n)}.
\end{equation}
Let us now prove \eqref{th.int.4}. Let $\beta = u \pi_L^k \in L^{\times}$, then $\operatorname{dlog\Bar{D}}_{L,v_n}(\beta) - \Psi_{L,v_n}(\beta)$ is equal to $u^{-1} \operatorname{\Bar{D}}_{L,v_n}(u) + k \pi_L^{-1} \operatorname{\Bar{D}}_{L,v_n}(\pi_L) -\Psi_{L,v_n}(u) -k \Psi_{L,v_n}(\pi_L)$ modulo $\pi_L^{-1}\mathfrak{X}_{L,1}^{(n)}$. However, by the very definition of $\operatorname{\Bar{D}}_{L,v_n}$, we have
\begin{equation} \label{th.int.5}
    \operatorname{\Bar{D}}_{L,v_n}(u) \equiv u \Psi_{L,v_n}(u) \mod{\mathfrak{X}_{L,1}^{(n)}}.
\end{equation}
But as $\mathfrak{X}_{L,1}^{(n)} \subset \pi_L^{-1}\mathfrak{X}_{L,1}^{(n)}$, the congruence \eqref{th.int.5} implies that 
\begin{equation}
    \operatorname{\Bar{D}}_{L,v_n}(u) \equiv u \Psi_{L,v_n}(u) \mod{\pi_L^{-1}\mathfrak{X}_{L,1}^{(n)}}.
\end{equation} 
Thus, we have
\begin{equation}
    u^{-1} \operatorname{\Bar{D}}_{L,v_n}(u) \equiv \Psi_{L,v_n}(u) \mod{\pi_L^{-1}\mathfrak{X}_{L,1}^{(n)}}.
\end{equation}
Moreover, we have 
\begin{equation}
    \operatorname{\Bar{D}}_{L,v_n}(\pi_L) \equiv \pi_L \Psi_{L,v_n}(\pi_L) \mod{\mathfrak{X}_{L,1}^{(n)}},
\end{equation}
and thus, 
\begin{equation}
    \pi_L^{-1} \operatorname{\Bar{D}}_{L,v_n}(\pi_L) \equiv \Psi_{L,v_n}(\pi_L) \mod{ \pi_L^{-1} \mathfrak{X}_{L,1}^{(n)}}.
\end{equation}
This concludes the proof.
\end{proof}
\subsection{Values of $\operatorname{\Bar{D}}_{L,v_n}$ in terms of representation theory} \label{sectionderivationrep}

Let $\mathcal{U}_K$ be the group of units of $K$. In this section, we will consider the continuous representation $\mathbf{r}:\operatorname{Gal}(\Omega|H) \to \operatorname{GL}_1(\mathcal{O}_K) = \mathcal{U}_K$ defined in \cite[Proposition 2.5]{Oukhaba}. The image $\mathbf{r}(\sigma)$ of an element $\sigma \in \operatorname{Gal}(\Omega|H)$ is the unique unit $u$ of $K$ such that $\sigma(\alpha) = \rho_{u}(\alpha)$ for all $\alpha \in W_{\rho}$. This representation is induced by the action of $\operatorname{Gal}(\Omega|H)$ on the module $\varprojlim W_{\rho}$. We will show that we can obtain explicit formulas in terms of invariants of this representation. 
It is obvious that the kernel of $\mathbf{r}$ is $\operatorname{Gal}(\Omega|H_{\rho})$. Thus, $\mathbf{r}$ induces an imbedding $\operatorname{Gal}(H_{\rho}|H) \to \mathcal{U}_K$. 
Reducing modulo $\mathcal{U}_{K,n}=1+\mathfrak{p}_K^n$, we get the map $\mathbf{r}_n$, which, restricted to $\operatorname{Gal}(H_{\rho}^n|H)$, defines an isomorphism
\begin{equation}
    \mathbf{r}_n:\operatorname{Gal}(H_{\rho}^n|H) \to \mathcal{U}_K / \mathcal{U}_{K,n}.
\end{equation}
For an algebraic extension $F$ of $H$, we also denote by $\mathbf{r}:\operatorname{Gal}(\Omega|F) \to \mathcal{U}_K$ the restriction of $\mathbf{r}$ to $\operatorname{Gal}(\Omega|F)$, and by $\mathbf{r}_n:\operatorname{Gal}(F(V_{\rho}^n)|F) \to \mathcal{U}_K / \mathcal{U}_{K,n}$ the restriction to $\operatorname{Gal}(F(V_{\rho}^n)|F)$.
\begin{prop} \label{propdefchi}
Let $m \ge n$ and suppose $L \supset E_{\rho}^m$. There exists a character $\chi_{L,m,n} : L^{\times} \to \mathcal{O}_K/\mathfrak{p}_K^{nm_0}$ such that 
\begin{equation}
    \mathbf{r}_{m_0(m+n)}(\Phi_L(\beta)) = 1 + \eta^m \chi_{L,m,n}(\beta) \in \mathcal{U}_K/\mathcal{U}_{K,(m+n)m_0} \label{defChi}
\end{equation}
for all $\beta \in L^{\times}$. Furthermore, $\chi_{L,m,n}$ satisfies the following.
\begin{enumerate}[label=(\roman*)]
    \item $\chi_{L,m,n}(\beta) = \chi_{E_{\rho}^m,m,n}(\operatorname{N}_{L|E_{\rho}^m}(\beta))$.
    \item Let $v=\rho_a(v_m) \in W_{\rho}^m$, where $v_m$ is a generator of $W_{\rho}^m$ such that $\rho_{\eta^{m-n}}(v_m)=v_n$. Then 
    \begin{equation*}
        (v,\beta)_{L,n} = (a \chi_{L,m,n}(\beta)) \cdot_{\rho} v_n
    \end{equation*}
    for all $\beta \in L^{\times}$. In particular if $v=v_m$, then for every $\beta \in L^{\times}$, we have
    \begin{equation} \label{eqpairingchi}
        (v_m,\beta)_{L,n} = \chi_{L,m,n}(\beta) \cdot_{\rho} v_n.
    \end{equation}
\end{enumerate}
\end{prop}

\begin{proof} 
Let $\beta \in L^{\times}$. As $\Phi_L(\beta)$ fixes $L$, thus in particular fixes $E_{\rho}^m$, we have
\begin{equation}
    \mathbf{r}_{m_0(m+n)}(\Phi_L(\beta)) \equiv 1 \mod \eta^m.
\end{equation}
Thus, there exists an element $\chi_{L,m,n}(\beta) \in \mathcal{O}_K/\mathfrak{p}_K^{nm_0}$ such that \eqref{defChi} holds. It is easy to check that $\chi_{L,m,n} : L^{\times} \to \mathcal{O}_K/\mathfrak{p}_K^{nm_0}$ is a group homomorphism. Moreover, the properties of the reciprocity map $\Phi_L$ imply (i). To prove (ii), let $\xi \in\mathfrak{p}_{\Omega}$ be such that $\rho_{\eta^n}(\xi)=v$. Such a $\xi$ exists by \cite[Lemma 2.1]{Papier1}. Since $v \in W_{\rho}^m$, then $\xi \in W_{\rho}^{m+n}$ and
\begin{align*}
    (v,\beta)_{L,n} & = \Phi_{L}(\beta)(\xi)-\xi \\
    & = \mathbf{r}_{m_0(m+n)}(\Phi_L(\beta))\cdot_{\rho} \xi - \xi \\
    & = (\mathbf{r}_{m_0(m+n)}(\Phi_L(\beta))-1) \cdot_{\rho} \xi \\
    & = (\eta^m \chi_{L,m,n}(\beta)) \cdot_{\rho} \xi \\
    & = (\eta^{m-n} \chi_{L,m,n}(\beta)) \cdot_{\rho} v \\
    & = (\eta^{m-n} \chi_{L,m,n}(\beta)) \cdot_{\rho} (a \cdot_{\rho} v_m) \\
    & = (a \chi_{L,m,n}(\beta)) \cdot_{\rho} v_n.
\end{align*}
\end{proof}
\begin{lemma} \label{chiisomstable}
The character $\chi_{L,m,n} : L^{\times} \to \mathcal{O}_K/\mathfrak{p}_K^{nm_0}$ is stable by isomorphism class of $\rho$. In other words, if $t$ is an invertible power series in $\mathcal{O}_H\{\{\tau\}\}$ such that $\rho'_a= t^{-1} \circ \rho_a \circ t$ for all $a \in \mathcal{O}_K$, then the characters defined in Proposition \ref{propdefchi} associated to $\rho$ and $\rho'$ are equal.
\end{lemma}

\begin{proof}
Let $v_m$ be such that $\rho_{\eta^{m-n}}(v_m)=v_n$ and let $v_i'=t^{-1}(v_i)$ for $i=m,n$. Denote by $\chi_{L,m,n}$ (respectively $\chi'_{L,m,n}$) the character defined in Proposition \ref{propdefchi} associated to $\rho$ (respectively $\rho'$). Then by \eqref{eqpairingchi},  we have $\chi'_{L,m,n}(\beta) \cdot_{\rho'} v'_n = (v'_m,\beta)_{\rho',L,n} $ which is equal to $t^{-1}((t(v'_m),\beta)_{\rho,L,n})=t^{-1}((v_m,\beta)_{\rho,L,n})$ by Proposition \ref{PropoertiesPairing1} (vi). Again from \eqref{eqpairingchi}, we conclude that
\begin{align*}
    \chi'_{L,m,n}(\beta) \cdot_{\rho'} v'_n & = t^{-1}((v_m,\beta)_{\rho,L,n}) \\
    & = t^{-1} (\chi_{L,m,n}(\beta) \cdot_{\rho} v_n) \\
    & = \chi_{L,m,n}(\beta) \cdot_{\rho'} t^{-1}(v_n) \\
    & =  \chi_{L,m,n}(\beta) \cdot_{\rho'} v'_n.
\end{align*}
\end{proof}
\begin{prop}
Let $m \ge n$ and suppose $L \supset E_{\rho}^m$ is such that $p$ does not divide the ramification index of the extension $L|E_{\rho}^m$. Let $u$ be a unit in $L$ such that $\mu(1-u) > \max\{\frac{nm_0}{q},\frac{1}{q-1}\} + \frac{1}{q-1}$. Let $f(X)$ and $g(X)$ be power series in $\mathbb{F}_{q_L}[[X]]$ such that $f(\pi_L)=v_m$ and $g(\pi_L)=u$. Then, 
\begin{equation} \label{propeqchi}
    \dfrac{\frac{g'(\pi_L)}{u}}{\frac{f'(\pi_L)}{v_m}} \in \mathfrak{p}_L,
\end{equation}
and
\begin{equation} \label{conguniq}
    \chi_{L,m,n}(u) \equiv \operatorname{T}_{L|K}((\frac{1-u}{u})(1- \dfrac{\frac{g'(\pi_L)}{u}}{\frac{f'(\pi_L)}{v_m}})\operatorname{\Bar{D}}_{L,v_n}(v_m)) \mod \mathfrak{p}_K^{nm_0}.
\end{equation}
\end{prop}

\begin{proof}
Since $p$ does not divide the ramification index of $L|E_{\rho}^m$, we have $\mu(f'(\pi_L)) = \mu(f(\pi_L)) - \mu(\pi_L) = \mu(v_m) - \mu(\pi_L)$. Furthermore, since $\mu(1-u) > \max\{\frac{nm_0}{q},\frac{1}{q-1}\} + \frac{1}{q-1}$, we can write $g(X)= 1 + \sum a_i X^i$, where $i \ge 2$ and $a_i \in \mathbb{F}_{q_L}$. Hence, $\mu(g'(\pi_L)) > \mu(\pi_L)$ and therefore, we have \eqref{propeqchi}. Now, let us prove \eqref{conguniq}.
By Lemma \ref{condrhoprime} and Lemma \ref{chiisomstable}, we can suppose that $\rho$ is such that $(x,x)_{\rho,L,n}=0$ for all $x \in \mathfrak{p}_L$. For such a $\rho$ and for $u \in L^{\times}$ such that $\mu(1-u) >\max\{\frac{nm_0}{q},\frac{1}{q-1}\}  + \frac{1}{q-1}$,  we have
\begin{equation} \label{interm3}
    (\alpha u, u)_{L,n} = \operatorname{T}_{L|K}((1-u)\operatorname{D}_{L,v_n}(\alpha)) \cdot_{\rho} v_n
\end{equation}
for all $\alpha \in \mathfrak{p}_L \setminus \{0\}$ by Lemma \ref{alphauu}. We note that the hypothesis on the valuation of $1-u$ allows us to replace $\operatorname{D}_{L,v_n}(\alpha)$ by $\operatorname{\Bar{D}}_{L,v_n}(\alpha)$ in \eqref{interm3}. Let $\alpha$ be such that $\alpha u = v_m$, where $v_m$ is a generator of $W_{\rho}^m$ such that $\rho_{\eta^{m-n}}(v_m)=v_n$. Hence, \eqref{interm3} together with \eqref{eqpairingchi} give us
\begin{equation}
    \chi_{L,m,n}(u) \cdot_{\rho} v_n = (v_m,u)_{L,n} = \operatorname{T}_{L|K}((1-u)\operatorname{\Bar{D}}_{L,v_n}(v_m u^{-1})) \cdot_{\rho} v_n.
\end{equation}
However,  $\operatorname{\Bar{D}}_{L,v_n}(\frac{v_m}{u}) = \frac{1}{u^2} (u \operatorname{\Bar{D}}_{L,v_n}(v_m) - v_m \operatorname{\Bar{D}}_{L,v_n}(u))$. Moreover, we have
\begin{equation*}
    \operatorname{\Bar{D}}_{L,v_n}(u) = g'(\pi_L) \operatorname{\Bar{D}}_{L,v_n}(\pi_L) ~~~\text{and} ~~~ \operatorname{\Bar{D}}_{L,v_n}(v_m) = f'(\pi_L) \operatorname{\Bar{D}}_{L,v_n}(\pi_L).
\end{equation*}
This implies that
\begin{equation*}
    f'(\pi_L) \operatorname{\Bar{D}}_{L,v_n}(u) - f'(\pi_L) g'(\pi_L) \operatorname{\Bar{D}}_{L,v_n}(\pi_L) \in \mathfrak{X}_{L,1}^{(n)}
\end{equation*}
and 
\begin{equation*}
   g'(\pi_L) \operatorname{\Bar{D}}_{L,v_n}(v_m) - f'(\pi_L) g'(\pi_L) \operatorname{\Bar{D}}_{L,v_n}(\pi_L) \in \mathfrak{X}_{L,1}^{(n)},
\end{equation*}
so that 
\begin{equation} \label{int.cong1}
    f'(\pi_L) \operatorname{\Bar{D}}_{L,v_n}(u) -  g'(\pi_L) \operatorname{\Bar{D}}_{L,v_n}(v_m) \in \mathfrak{X}_{L,1}^{(n)}.
\end{equation}
Now, since the calculation in the beginning of this proof shows that $\frac{v_m}{f'(\pi_L)} \in \mathfrak{p}_L$, we can multiply \eqref{int.cong1} by $\frac{v_m}{f'(\pi_L)}$ in the fractional ideal $\mathfrak{X}_{L,1}^{(n)}$. Therefore, we get
\begin{equation*}
    v_m \operatorname{\Bar{D}}_{L,v_n}(u) = v_m \dfrac{ g'(\pi_L)}{f'(\pi_L)} \operatorname{\Bar{D}}_{L,v_n}(v_m) \in \mathfrak{X}_{L,1} / \mathfrak{X}_{L,1}^{(n)}.
\end{equation*}
Finally, we can write
\begin{align*}
    \chi_{L,m,n}(u) & \equiv \operatorname{T}_{L|K}(\frac{1-u}{u^2}( u \operatorname{\Bar{D}}_{L,v_n}(v_m) - v_m \frac{g'(\pi_L)}{f'(\pi_L)} \operatorname{\Bar{D}}_{L,v_n}(v_m)) \mod \mathfrak{p}_K^{nm_0} \\
    & \equiv \operatorname{T}_{L|K}((\frac{1-u}{u})(1- \dfrac{\frac{g'(\pi_L)}{u}}{\frac{f'(\pi_L)}{v_m}})\operatorname{\Bar{D}}_{L,v_n}(v_m)) \mod \mathfrak{p}_K^{nm_0}.
\end{align*}
\end{proof}
\begin{lemma} \label{lemmaUniqdetermined}
Let $m \ge n$ and suppose $L \supset E_{\rho}^m$ is such that $p$ does not divide the ramification index of the extension $L|E_{\rho}^m$. Then,  $\operatorname{\Bar{D}}_{L,v_n}(v_m) \in \mathfrak{X}_{L,1} / \frac{\eta^n}{\gamma} \mathfrak{X}_{L,1}$ is uniquely determined by \eqref{conguniq}.
\end{lemma}

\begin{proof}
Let $x$ and $x'$ be two elements in $\mathfrak{X}_{L,1}$ such that
\begin{equation}
    \operatorname{T}_{L|K}((\frac{1-u}{u})(1- \dfrac{\frac{g'(\pi_L)}{u}}{\frac{f'(\pi_L)}{v_m}})x) \equiv \operatorname{T}_{L|K}((\frac{1-u}{u})(1- \dfrac{\frac{g'(\pi_L)}{u}}{\frac{f'(\pi_L)}{v_m}})x') \mod \mathfrak{p}_K^{nm_0},
\end{equation}
for all $u \in \mathcal{U}_L$ such that $\mu(1-u) > \max\{\frac{nm_0}{q},\frac{1}{q-1}\}  + \frac{1}{q-1}$. This means that
\begin{equation} \label{interm4}
    \operatorname{T}_{L|K}((\frac{1-u}{u})(1- \dfrac{\frac{g'(\pi_L)}{u}}{\frac{f'(\pi_L)}{v_m}})(x-x')) \in \mathfrak{p}_K^{nm_0}
\end{equation}
for all units $u \in L$ such that $\mu(1-u) > \max\{\frac{nm_0}{q},\frac{1}{q-1}\}  + \frac{1}{q-1}$. We need to prove that $x-x' \in  \mathfrak{X}_{L,1}^{(n)}$. Since we are considering any $u$ such that $\mu(1-u) > \max\{\frac{nm_0}{q},\frac{1}{q-1}\}  + \frac{1}{q-1}$, then we can write $1-u = \gamma \alpha$, where $\gamma \in L$ is of valuation $\max\{\frac{nm_0}{q},\frac{1}{q-1}\}$ and $\alpha$ varies in $\mathfrak{p}_{L,1} = \lambda_{\rho}(\mathfrak{p}_{L,1})$. 
Furthermore, the element $1- \dfrac{\frac{g'(\pi_L)}{u}}{\frac{f'(\pi_L)}{v_m}}$ is a unit in $L$. Therefore, $x$ and $x'$ are such that
\begin{equation}
  \operatorname{T}_{L|K}(\gamma \alpha (x-x')) \in \mathfrak{p}_K^{nm_0} 
\end{equation}
for all $\alpha \in \mathfrak{p}_{L,1}$. This yields that $x-x' \in \frac{\eta^n}{\gamma} \mathfrak{X}_{L,1} = \mathfrak{X}_{L,1}^{(n)}$.
\end{proof}
\subsection{Explicit formulas in a particular case}
In this section, we place ourselves in the case where  $ \rho_{\eta} \equiv \tau^{m_0} \mod \mathfrak{p}_H$ and $L = E_{\rho}^m \supset E_{\rho}^n$ for an integer $m \ge n$. As previously shown in Theorem \ref{mainTH}, we have
\begin{equation} \label{resultPaper2}
(\alpha,v_m)_{L,n}=  T_{L|K}(\lambda_{\rho}(\alpha)\frac{1}{v_m}\operatorname{\Bar{D}}_{L,v_n}(v_m)) \cdot_{\rho} v_n
\end{equation}
for all $\alpha \in \mathfrak{p}_L$ such that $\mu(\alpha) \ge\max\{\frac{nm_0}{q},\frac{1}{q-1}\}+\frac{1}{q-1}+ \frac{1}{q^{nm_0}(q-1)}$. On the other hand, we can prove using \cite{Papier1}, that for the same condition on $\alpha$, we have
\begin{equation} \label{resultPaper1}
     (\alpha,v_m)_{L,n}= \frac{1}{\eta^m} T_{L|K}(\lambda_{\rho}(\alpha)\frac{1}{v_m}) \cdot_{\rho} v_n.
\end{equation}
Indeed,
\begin{align*}
    (\alpha,v_m)_{L,n} & = (\rho_{\eta^{m-n}}(\alpha),v_m)_{L,m} \quad \quad \quad \quad \quad \quad  \quad \quad \quad  \text{(by \cite[Proposition 2.2]{Papier1})}\\
     & = \frac{1}{\eta^m} T_{L|K}(\lambda_{\rho}(\rho_{\eta^{m-n}}(\alpha))\frac{1}{v_m}) \cdot_{\rho} v_m \quad \quad  \quad   \text{(by \cite[Theorem 5.7]{Papier1})}\\
     & = \frac{1}{\eta^m} T_{L|K}(\eta^{m-n}\lambda_{\rho}(\alpha)\frac{1}{v_m}) \cdot_{\rho} v_m\\
     & = \frac{1}{\eta^m} T_{L|K}(\lambda_{\rho}(\alpha)\frac{1}{v_m}) \cdot_{\rho} v_n.
\end{align*}
Here, we can apply \cite[Theorem 5.7]{Papier1} for $(\rho_{\eta^{m-n}}(\alpha),v_m)_{L,m}$ because $\mu(\rho_{\eta^{m-n}}(\alpha)) \ge \max\{\frac{mm_0}{q},\frac{1}{q-1}\}+\frac{1}{q-1}+ \frac{1}{q^{mm_0}(q-1)}$ for all $m \ge n$ (see proof of \cite[Lemma 4.1]{Papier1}).
\begin{prop} \label{propValueD}
We have
\begin{equation} \label{valueD}
    \operatorname{\Bar{D}}_{E_{\rho}^m,v_n}(v_m) = \frac{1}{\eta^m}.
\end{equation}   
\end{prop}

\begin{proof}
    This is a direct consequence of the explicit formulas \eqref{resultPaper2} and \eqref{resultPaper1}.
\end{proof}

\begin{corollary} \label{coroCong}
Let $u$ be a unit of $L$ such that $\mu(1-u) > \max\{\frac{nm_0}{q},\frac{1}{q-1}\} + \frac{1}{q-1}$. Then
\begin{equation} \label{lc2}
    \operatorname{N}_{L|K}(u^{-1})-1 \equiv \operatorname{T}_{L|K}((\frac{1-u}{u})(1- \frac{g'(v_m)}{u}v_m))  \mod \mathfrak{p}_K^{(n+m)m_0},
\end{equation}
where $g(X) \in \mathbb{F}_{q_L}[[X]]$ is such that $g(v_m)=u$.
\end{corollary}

\begin{proof}
Since we proved in Lemma \ref{lemmaUniqdetermined} that $\Bar{\operatorname{D}}_{L,v_n}(v_m)= \frac{1}{\eta^m} \in \mathfrak{X}_{L,1} / \mathfrak{X}_{L,1}^{(n)}$ is uniquely determined by \eqref{conguniq}, we have
\begin{equation} \label{lc1}
    \chi_{L,m,n}(u) \equiv \frac{1}{\eta^m} \operatorname{T}_{L|K}((\frac{1-u}{u})(1- \frac{g'(v_m)}{u}v_m)) \mod \mathfrak{p}_K^{nm_0}
\end{equation}
for all units $u$ of $L$ such that $\mu(1-u) > \max\{\frac{nm_0}{q},\frac{1}{q-1}\} + \frac{1}{q-1}$.
Moreover, we know by \eqref{eqpairingchi} that 
\begin{equation} 
        (v_m,u)_{L,n} = \chi_{L,m,n}(u) \cdot_{\rho} v_n =  \eta ^m \chi_{L,m,n}(u) \cdot_{\rho} v_{m+n},
\end{equation}
where $v_{m+n} \in W_{\rho}^{m+n}$ is such that $\rho_{\eta^n}(v_{m+n})=v_m$. On the other hand, by the definition of $(~,~)_{L,n}$, we have
\begin{equation}
    (v_m,u)_{L,n}= \Phi_L(u)(v_{m+n})-v_{m+n} = \Phi_K(\operatorname{N}_{L|K}(u))(v_{m+n})-v_{m+n}.
\end{equation}
But $ \Phi_K(\operatorname{N}_{L|K}(u))(v_{m+n})= \rho_{\operatorname{N}_{L|K}(u)^{-1}}(v_{m+n})$ by \cite[Proposition 5.1]{Papier1}. Therefore, $(v_m,u)_{L,n}= (\operatorname{N}_{L|K}(u^{-1})-1)\cdot_{\rho} v_{m+n}$ and hence, 
\begin{equation} \label{lc5}
    \operatorname{N}_{L|K}(u^{-1})-1 \equiv \eta ^m \chi_{L,m,n}(u) \mod \mathfrak{p}_K^{(n+m)m_0}.
\end{equation}
Finally, \eqref{lc2} follows from \eqref{lc1} and \eqref{lc5}.
\end{proof}
\bibliographystyle{plain}
\bibliography{biblio.bib}

@article {Ignazio,
    AUTHOR = {Bars, F. and Longhi, I.},
     TITLE = {Coleman's power series and {W}iles' reciprocity for rank 1
              {D}rinfeld modules},
   JOURNAL = {J. Number Theory},
  FJOURNAL = {Journal of Number Theory},
    VOLUME = {129},
      NUMBER = {4},
     PAGES = {789--805},
     YEAR = {2009}}

@article {Oukhaba,
    AUTHOR = {Oukhaba, H.},
     TITLE = {On local fields generated by division values of formal
              {D}rinfeld modules},
   JOURNAL = {Glasg. Math. J.},
  FJOURNAL = {Glasgow Mathematical Journal},
    VOLUME = {62},
      YEAR = {2020},
    NUMBER = {2},
     PAGES = {459--472}}

@article {Rosen,
    AUTHOR = {Rosen, M.},
     TITLE = {Formal {D}rinfeld modules},
   JOURNAL = {J. Number Theory},
  FJOURNAL = {Journal of Number Theory},
    VOLUME = {103},
      YEAR = {2003},
    NUMBER = {2},
     PAGES = {234--256}}

@book {Iwasawa,
    AUTHOR = {Iwasawa, K.},
     TITLE = {Local class field theory},
    SERIES = {Oxford Science Publications},
      NOTE = {Oxford Mathematical Monographs},
 PUBLISHER = {The Clarendon Press, Oxford University Press, New York},
      YEAR = {1986},
     PAGES = {viii+155}}

@article {Iwasawa2,
    AUTHOR = {Iwasawa, K.},
     TITLE = {On some modules in the theory of cyclotomic fields},
   JOURNAL = {J. Math. Soc. Japan},
  FJOURNAL = {Journal of the Mathematical Society of Japan},
    VOLUME = {16},
      YEAR = {1964},
     PAGES = {42--82}}

@article {Wiles,
    AUTHOR = {Wiles, A.},
     TITLE = {Higher explicit reciprocity laws},
   JOURNAL = {Ann. of Math. (2)},
  FJOURNAL = {Annals of Mathematics. Second Series},
    VOLUME = {107},
      YEAR = {1978},
    NUMBER = {2},
     PAGES = {235--254}}

@article {kolyvagin,
    AUTHOR = {Kolyvagin, V. A.},
     TITLE = {Formal groups and the norm residue symbol},
   JOURNAL = {Izv. Akad. Nauk SSSR Ser. Mat.},
  FJOURNAL = {Izvestiya Akademii Nauk SSSR. Seriya Matematicheskaya},
    VOLUME = {43},
      YEAR = {1979},
    NUMBER = {5},
     PAGES = {1054--1120, 1198}}

@article {angles,
    AUTHOR = {Angl\`es, B.},
     TITLE = {On explicit reciprocity laws for the local {C}arlitz-{K}ummer
              symbols},
   JOURNAL = {J. Number Theory},
  FJOURNAL = {Journal of Number Theory},
    VOLUME = {78},
      YEAR = {1999},
    NUMBER = {2},
     PAGES = {228--252}}

@article {Florez1,
    AUTHOR = {Fl\'{o}rez, J.},
     TITLE = {Explicit reciprocity laws for higher local fields},
   JOURNAL = {J. Number Theory},
  FJOURNAL = {Journal of Number Theory},
    VOLUME = {213},
      YEAR = {2020},
     PAGES = {400--444}}

@article{Florez2,
    author = {Fl\'{o}rez, J.},
    title = {The norm residue symbol for higher local fields},
    journal = {Journal of Number Theory},
    year = {2021, in press},
    issn = {0022-314X},
    doi = {https://doi.org/10.1016/j.jnt.2021.06.031},
    url = {https://www.sciencedirect.com/science/article/pii/S0022314X21002341}}

@article{Papier1,
author = {Ala Eddine, M.},
year = {2024},
month = {03},
pages = {675-695},
title = {Explicit Reciprocity Laws for Formal Drinfeld Modules},
volume = {35},
journal = {Journal de théorie des nombres de Bordeaux}
}
\end{document}